\documentclass[a4paper,12pt,draft]{article}

\usepackage{amssymb}
\usepackage{latexsym}
\usepackage{amsfonts}
\usepackage{amsmath}
\usepackage{amsthm}
\usepackage{mathrsfs}
\usepackage{hyperref}
\usepackage{tensor}
\usepackage{mathtools}
\usepackage{mdwlist}
\usepackage{pgf,tikz,pgfplots}
\pgfplotsset{width=7.5cm,compat=1.9}

\usetikzlibrary{arrows,patterns,positioning}
\usetikzlibrary{calc}

\usepackage[margin=1in]{geometry}

\usepackage[compact]{titlesec}

\usepackage[scale=0.9]{tgheros}
\usepackage[T1]{fontenc}

\DeclareMathOperator{\Aut}{Aut}

\DeclareMathOperator{\SL}{SL}

\DeclareMathOperator{\PSU}{PSU}

\DeclareMathOperator{\pg}{PG}
\DeclareMathOperator{\PGaU}{P\Gamma U}
\DeclareMathOperator{\PSp}{PSp}
\DeclareMathOperator{\PGaSp}{P\Gamma Sp}
\DeclareMathOperator{\PGaO}{P\Gamma O}

\DeclareMathOperator{\Sp}{Sp}

\DeclareMathOperator{\Sz}{Sz}

\DeclareMathOperator{\PGaL}{P\Gamma L}

\DeclareMathOperator{\GL}{GL}
\DeclareMathOperator{\GaL}{\Gamma L}

\DeclareMathOperator{\I}{I}
\DeclareMathOperator{\alt}{A}
\DeclareMathOperator{\s}{S}

\renewcommand{\leq}{\leqslant}
\renewcommand{\geq}{\geqslant}

\newcommand{\F}{\mathbb F}

\newcommand{\Q}{\mathcal Q}

\newcommand{\Qu}{\mathsf Q}
\newcommand{\W}{\mathsf W}
\renewcommand{\L}{\mathcal L}
\renewcommand{\H}{\mathsf H}
\renewcommand{\P}{\mathcal P}
\renewcommand{\S}{\mathcal S}
\renewcommand{\O}{\mathcal O}

\theoremstyle{plain}
\newtheorem{lemma}{Lemma}
\newtheorem{theorem}[lemma]{Theorem}
\newtheorem{proposition}[lemma]{Proposition}

\theoremstyle{definition}
\newtheorem{definition}[lemma]{Definition}
\newtheorem{example}[lemma]{Example}
\newtheorem{conjecture}[lemma]{Conjecture}
\newtheorem{question}[lemma]{Question}

\numberwithin{equation}{section}
\numberwithin{lemma}{section}

\begin{document}

\title{Neighbour-Transitive Codes and Partial Spreads in Generalised Quadrangles\thanks{This work has been supported by the Croatian Science Foundation under the projects 6732 and 5713.}}
\author{Dean Crnkovi{\'c}, Daniel R. Hawtin and Andrea {\v S}vob}
\date{
 \small{
  \emph{
   Department of Mathematics, University of Rijeka\\
   Rijeka, Croatia, 51000\\
  }
  \vspace{0.25cm}
  \href{mailto:deanc@math.uniri.hr}{deanc@math.uniri.hr}\hspace{0.5cm} \href{mailto:dan.hawtin@gmail.com}{dan.hawtin@gmail.com}\hspace{0.5cm} \href{mailto:asvob@math.uniri.hr}{asvob@math.uniri.hr}\\
  \vspace{0.25cm}
 }
\today
}

\maketitle

\begin{abstract}
 A code $C$ in a generalised quadrangle $\Q$ is defined to be a subset of the vertex set of the point-line incidence graph $\varGamma$ of $\Q$. The minimum distance $\delta$ of $C$ is the smallest distance between a pair of distinct elements of $C$. The graph metric gives rise to the distance partition $\{C,C_1,\ldots,C_\rho\}$, where $\rho$ is the maximum distance between any vertex of $\varGamma$ and its nearest element of $C$. Since the diameter of $\varGamma$ is $4$, both $\rho$ and $\delta$ are at most $4$. If $\delta=4$ then $C$ is a partial ovoid or partial spread of $\Q$, and if, additionally, $\rho=2$ then $C$ is an ovoid or a spread. A code $C$ in $\Q$ is neighbour-transitive if its automorphism group acts transitively on each of the sets $C$ and $C_1$. Our main results i) classify all neighbour-transitive codes admitting an insoluble group of automorphisms in thick classical generalised quadrangles that correspond to ovoids or spreads, and ii) give two infinite families and six sporadic examples of neighbour-transitive codes with minimum distance $\delta=4$ in the classical generalised quadrangle $\W_3(q)$ that are not ovoids or spreads. 
\end{abstract}

\section{Introduction}

Important generalisations of perfect codes were introduced in the 1970s, namely Semakov, Zinoviev and Zaitsev \cite{SemZinZai71} introduced \emph{uniformly packed codes} and Delsarte \cite{delsarte1973algebraic} defined the classes of \emph{completely regular} and \emph{$s$-regular} codes. Although it has been typical to study codes in the context of the Hamming or Johnson graphs, authors commonly study codes as subsets of the vertex set of more general classes of graphs (see for example \cite{Biggs1973289}). In this paper, we consider codes in the point-line incidence graphs of generalised quadrangles, in particular, those that satisfy certain symmetry conditions.

Let $\varGamma$ be a simple connected graph on finitely many vertices. For vertices $\alpha$ and $\beta$ of $\varGamma$ denote the graph distance between $\alpha$ and $\beta$ by $d(\alpha,\beta)$. A \emph{code} $C$ in $\varGamma$ is a subset of the vertex set of $\varGamma$, the elements of which are called \emph{codewords}. The \emph{minimum distance} $\delta$ of $C$ is the smallest distance, in the graph $\varGamma$, between two distinct codewords of $C$. The \emph{set of $i$-neighbours} of $C$, denoted $C_i$, is the subset of all those vertices $\gamma$ of $\varGamma$ for which there exists $\alpha\in C$ such that $d(\alpha,\gamma)=i$ and $d(\beta,\gamma)\geq i$ for all $\beta\in C$. The \emph{distance partition} of $C$ is then the partition $\{C,C_1,C_2,\ldots,C_\rho\}$ of the vertex set of $\varGamma$, where $\rho$ is the \emph{covering radius} of $C$, that is, the maximum value of $i$ for which $C_i$ is non-empty. The automorphism group $\Aut(C)$ of $C$ is the setwise stabiliser of $C$ inside $\Aut(\varGamma)$. We consider the following symmetry conditions:

\begin{definition}\label{sneighbourtransdef}
 Let $C$ be a code in a graph $\varGamma$ with covering radius $\rho$, and let $i\in\{1,\ldots,\rho\}$. Then $C$ is said to be
 \begin{enumerate}
  \item \emph{$i$-neighbour-transitive} if $\Aut(C)$ acts transitively on each of the sets $C,C_1,\ldots, C_i$,
  \item \emph{neighbour-transitive} if $C$ is $1$-neighbour-transitive, and,
  \item \emph{completely transitive} if $C$ is $\rho$-neighbour-transitive.
 \end{enumerate}
\end{definition}

Neighbour-transitive codes in Hamming graphs were first studied in Gillespie's PhD thesis \cite{neilphd}, as well as \cite{ntrcodes,gillespieCharNT,gillespiediadntc}. Investigating $2$-neighbour-transitive codes in Hamming graphs constitutes a major part of the second author's PhD thesis \cite{myphdthesis}; see also \cite{ef2nt,aas2nt,ondimblock,minimal2nt}. Completely transitive codes are a subclass of completely regular codes (see \cite{borges2019completely}) and completely transitive codes in Hamming graphs have been the subject of \cite{bailey2020classification,borges2001nonexistence,borges2010q,Borges201468,borges2019completely,Giudici1999647}. Neighbour-transitive codes in Johnson graphs have been considered in \cite{durante2014sets,liebler2014neighbour,neunhoffer2014sporadic} as well as Ioppolo's PhD thesis \cite{marksphd}. Note that there are open problems in all of the aforementioned areas of study.

From this point forward, we specialise to the case that $\varGamma$ is the point-line incidence graph of a generalised quadrangle $\Q$ and we say that a code in $\varGamma$ is a code in the generalised quadrangle $\Q$ (see Section~\ref{prelim} for background on generalised quadrangles). Note that since the diameter of $\varGamma$ is $4$ it follows that for any such code $C$ we have $\delta,\rho\leq 4$. We denote by $\Aut(\varGamma)$ the automorphism group of $\varGamma$ and note that the automorphism group $\Aut(\Q)$ of $\Q$ is a subgroup of $\Aut(\varGamma)$, with this inclusion being proper if and only if $\Q$ is isomorphic to its dual, in which case $\Aut(\Q)$ has index $2$ in $\Aut(\varGamma)$. 

Our main result, stated below, concerns neighbour-transitive codes in the classical generalised quadrangle $\W_3(q)$ (see Section~\ref{prelim} for the definition of $\W_3(q)$ and for more information concerning spreads of generalised quadrangles). 

\begin{theorem}\label{mainresult}
 Let $C$ be a neighbour-transitive code with minimum distance $4$ in the generalised quadrangle $\W_3(q)$. Then the following hold:
 \begin{enumerate}
  \item $C$ is equivalent to a regular spread if and only if $C$ has covering radius $2$.
  \item If $|C|=q^2$ then $C$ has covering radius $\rho=4$ and is equivalent to a spread minus a line. 
  \item If $|C|=q^2-1$ then $q=2,3,5,7$ or $11$ and $C$ is equivalent to one of the codes in Example~\ref{maxspreadconstr}, each of which has covering radius $\rho=3$.
  \item If $|C|=q+1$ and $C$ has covering radius $\rho=3$ then $C$ is equivalent to the set of points on a hyperbolic line.
  \item If $q=3$ and $|C|=5$ then $C$ has covering radius $\rho=3$ and is equivalent to the code given in Example~\ref{smallqequals3}.
 \end{enumerate}
\end{theorem}

Note that an example of a code as in part 2 of Theorem~\ref{mainresult} is given, for each $q$, by a regular spread of $\W_3(q)$ with one line removed; see Lemma~\ref{regspreadminusline}. We do not prove here that these are all such examples. However, since the automorphism group of such a code must have order divisible by $q^2(q+1)$, \cite[Table~1]{penttila2000ovoids} suggests that perhaps this is the case. Hence, we pose the following question.

\begin{question}
 Do there exist codes as in part 2 of Theorem~\ref{mainresult} that are not contained in a regular spread?
\end{question}

Lemma~\ref{ovoidorspread} shows that a code having minimum distance $\delta=4$ and covering radius $\rho\leq 3$ is a maximal partial ovoid or a maximal partial spread. In regards to neighbour-transitive maximal partial ovoids and maximal partial spreads in $\W_3(q)$ we make the following conjecture, which has been confirmed up to $q=7$ by computations performed in MAGMA \cite{MR1484478} and in GAP \cite{GAP4} using the package FinInG \cite{fining}.

\begin{conjecture}
 Let $C$ be a neighbour-transitive maximal partial ovoid or maximal partial spread of $\W_3(q)$. Then $C$ is as in Theorem~\ref{mainresult} parts 1,3,4 or 5.
\end{conjecture}

In Theorem~\ref{noperfectcodes} we prove that there are no perfect codes in generalised quadrangles that admit an automorphism group acting transitively on the code. The proof of Theorem~\ref{noperfectcodes} implies that a code with minimum distance $\delta=3$ in a generalised quadrangle of order $s$ has size at most $2(s^2-s+1)$ and we leave open the question as to whether such a code exists.

\begin{question}\label{maxcodedelta3}
 Does there exist a code of size $2(s^2-s+1)$ with minimum distance $3$ in a generalised quadrangle of order $s$?
\end{question}

In the next section we introduce the concepts and terminology we shall use regarding generalised quadrangles and codes. In Section~\ref{codessect} we consider, more generally, codes in the point-line incidence graphs of generalised quadrangles. In Section~\ref{ovoidssect} we classify neighbour-transitive ovoids and spreads in thick classical generalised quadrangles, and in Section~\ref{w3qsect} we focus on neighbour-transitive partial ovoids and partial spreads in the classical generalised quadrangle $\W_3(q)$, proving Theorem~\ref{mainresult}.

\section{Preliminaries}\label{prelim}

A \emph{generalised quadrangle} is an incidence structure $\Q=(\P,\L,\I)$, where $\P$ and $\L$ are disjoint non-empty sets called points and lines, respectively, and $\I$ is a symmetric point-line incidence relation such that:
\begin{enumerate}
 \item Each point is incident with $t+1$ lines ($t\geq 1$) and two distinct points are incident with at most one line.
 \item Each line is incident with $s+1$ points ($s\geq 1$) and two distinct lines are incident with at most one point.
 \item If $p$ is a point and $L$ is a line not incident with $p$, then there is a unique pair $(q,M)\in\P\times \L$ for which $p\,\I\,M\,\I\,q\,\I\,L$.
\end{enumerate}
For more comprehensive background on generalised quadrangles than is covered here we refer the reader to \cite{paynefinite}. A generalised quadrangle $\Q$ satisfying the above axioms is said to have order $(s,t)$, or simply order $s$ when $s$ and $t$ are equal. A generalised quadrangle is called \emph{thick} if $s,t\geq 2$. \emph{Classical} generalised quadrangles are associated with certain simple groups; see \cite[Chapter 3]{paynefinite} for their constructions and \cite[Sections~3.5.6 and 3.6.4]{wilson2009finite} for more about their automorphism groups; see also Table~\ref{table:classGQs}. Note that the \emph{dual} $\Q^D$ of a generalised quadrangle $\Q$ of order $(s,t)$, obtained by interchanging the role of points and lines (that is, $\Q^D=(\L,\P,\I)$), is a generalised quadrangle having order $(t,s)$. Isomorphisms of generalised quadrangles are defined as incidence preserving bijections between points/lines. A generalised quadrangle is said to be \emph{self-dual} if it is isomorphic to its dual.

We are interested here, in particular, with the classical generalised quadrangle $\W_3(q)$, which has order $q$. If $V$ is the underlying vector space of the projective geometry $\pg_3(q)$ equipped with a non-degenerate symplectic bilinear form $f$ then points of $\W_3(q)$ are the points of $\pg_3(q)$ and the lines of $\W_3(q)$ are the totally isotropic $2$-dimensional subspaces of $V$, that is, the $2$-dimensional subspaces $X\leq V$ such that $f(x,y)=0$ for all $x,y\in X$; incidence is given by symmetrised containment. If $q$ is even then $\W_3(q)$ is self dual, while if $q$ is odd then the dual of $\W_3(q)$ is $\Qu_4(q)$.

Let $\Q$ be a generalised quadrangle. Then the point-line incidence graph $\varGamma$ of $\Q$ is the graph with vertex set $V\varGamma=\P\cup \L$ with adjacency given by incidence. The graph $\varGamma$ is bipartite, has degrees $s+1$ and $t+1$, diameter $4$ and girth $8$. Let $d$ be the graph metric. For a vertex $\alpha\in\varGamma$, we denote the set of $i$-neighbours of $\alpha$ by $\varGamma_i(\alpha)=\{\beta\in V\varGamma\mid d(\alpha,\beta)=i\}$. As in the introduction, a code in a generalised quadrangle is subset of $V\varGamma$, the minimum distance and covering radius of a code are also defined there. We say that a code $C$ is \emph{trivial} if either $|C|\leq 1$ or $C$ consists of the entire vertex set of the graph in which it is defined. Two codes are \emph{equivalent} if there exists an automorphism of $\varGamma$ mapping one to the other. Note that equivalent codes have the same minimum distance, covering radius and isomorphic automorphism groups.

An \emph{ovoid} of a generalised quadrangle $\Q=(\P,\L,\I)$ is a subset $\O$ of $\P$ such that each line $\ell\in\L$ is incident with exactly one point of $\O$, and a \emph{partial ovoid} is a subset $\O$ of $\P$ such that each line $\ell\in\L$ is incident with at most one point of $\O$. Dually, a \emph{spread} of a generalised quadrangle $\Q=(\P,\L,\I)$ is a subset $\S$ of $\L$ such that each point $p\in\P$ is incident with exactly one line of $\S$, and a \emph{partial spread} is a subset $\S$ of $\L$ such that each point $p\in\P$ is incident with at most one line of $\S$. A spread or an ovoid is called \emph{maximal} if it is not contained in a larger spread or ovoid, respectively. A spread is \emph{regular} or \emph{Desarguesian} if it may be obtained via \emph{field reduction}, that is, comes from an embedding of $\pg_1(q^2)$ into $\pg_3(q)$.

%

\begin{table}
 \begin{center}
 \begin{tabular}{cccc}
  $\Q$ & order & $\Q^D$ & $\Aut(\Q)$ \\
  \hline
  ${\mathsf Q}_4(q)$ & $(q,q)$ & ${\mathsf W}_3(q)$ & $\PGaO_5(q)$ \\
  ${\mathsf Q}^-_5(q)$ & $(q,q^2)$ & ${\mathsf H}_3(q^2)$ & $\PGaO^-_6(q)$ \\
  ${\mathsf W}_3(q)$ & $(q,q)$ & ${\mathsf Q}_4(q)$ & $\PGaSp_4(q)$ \\
  ${\mathsf H}_3(q^2)$ & $(q^2,q)$ & ${\mathsf Q}^-_5(q)$ & $\PGaU_4(q)$ \\
  ${\mathsf H}_4(q^2)$ & $(q^2,q^3)$ & ${\mathsf H}_4(q^2)^D$ & $\PGaU_5(q)$ \\
  \hline
 \end{tabular}
 \caption{Thick classical generalised quadrangles, their orders, duals, and automorphism groups. Note that $W_3(q)$ is self-dual when $q$ is even. }
 \label{table:classGQs}
 \end{center}
\end{table}

\section{Codes in generalised quadrangles}\label{codessect}

Let $\varGamma$ be the point-line incidence graph of a generalised quadrangle $\Q$ and let $C$ be a code in $\Q$. Since the diameter of $\varGamma$ is $4$, we have that the covering radius of $C$ satisfies $\rho\leq 4$. In Sections~\ref{ovoidssect} and~\ref{w3qsect} we focus on codes $C$ with minimum distance $4$, in which case $C$ is a partial ovoid or a partial spread of $\Q$; see Lemma~\ref{ovoidorspread}. However, we first consider codes with smaller minimum distance. Note that a neighbour-transitive code $C$ necessarily admits an automorphism group that acts transitively on $C$ and also that the only known self-dual generalised quadrangles are the translation quadrangles $T_2(\O)$ (see \cite{payne2012essay}).

\begin{lemma}\label{selfdaullemma}
 Let $C$ be a code in a generalised quadrangle $\Q$ with odd minimum distance $\delta$ and suppose that $\Aut(C)$ acts transitively on $C$. Then $\Q$ is self-dual.
\end{lemma}

\begin{proof}
 Since $\delta$ is odd, there exists a point $p$ and a line $\ell$ of $\Q$ such that $p,\ell\in C$. Since $\Aut(C)$ acts transitively on $C$, it follows that there exists an automorphism of $\varGamma$ interchanging the set of points and the set of lines of $\Q$. Hence $\Q$ is self-dual.
\end{proof}

A code $C$ is \emph{perfect} if the covering radius $\rho$ of $C$ satisfies $\rho=e$, where $e$ is the \emph{error-correction capacity} of $C$, that is,
\[
 \rho=e=\left\lfloor \frac{\delta-1}{2}\right\rfloor.
\]
This is equivalent to the condition that the set of balls of radius $\rho$ centred at each codeword of $C$ form a partition of the vertex set of $\varGamma$, which immediately implies that $\delta$ is odd. We have the following theorem.

\begin{theorem}\label{noperfectcodes}
 There are no non-trivial perfect codes in generalised quadrangles admitting an automorphism group that acts transitively on the code.
\end{theorem}

\begin{proof}
 Suppose $C$ is a perfect code in a generalised quadrangle $\Q$ such that $\Aut(C)$ acts transitively on $C$. Since a perfect code necessarily has odd minimum distance, and the diameter of the point-line incidence graph $\varGamma$ is $4$, we have that either $\delta=1$ and $\rho=0$, or $\delta=3$ and $\rho=1$. If $\delta=1$ and $\rho=0$ then $C$ is trivial. Lemma~\ref{selfdaullemma} implies that $\Q$ is self-dual, and hence has order $s\geq 1$. In particular, $\Q$ has $s^3+s^2+s+1$ points, and the same number of lines, so that $\varGamma$ has $2(s^3+s^2+s+1)$ vertices. Since $C$ is perfect, $\{C,C_1\}$ forms a partition of the vertex set of $\varGamma$ so that
 \[
  |C|+|C_1|=2(s^3+s^2+s+1).
 \]
 Since each element of $C_1$ is incident with a unique element of $C$, and each element of $C$ is incident with $s+1$ elements of $C_1$, the left-hand side of the above equation becomes $|C|(s+2)$. Thus
 \[
  |C|=2\frac{s^3+s^2+s+1}{s+2}=2(s^2-s+1)-\frac{2}{s+2}.
 \]
 Since the size of the code must be an integer, this implies that $s=0$, giving a contradiction.
\end{proof}



We will often use the following result, or its dual, without reference in order to prove that a code $C$ is neighbour-transitive.

\begin{lemma}\label{lemma:Xptrans}
 Let $C$ be a code in a generalised quadrangle $\Q$ having minimum distance $\delta=3$ or $4$ and assume the point $p$ is in $C$. Then $C$ is neighbour-transitive if and only if $\Aut(C)$ acts transitively on $C$ and the stabiliser $\Aut(C)_p$ of $p$ acts transitively on the set of all lines incident with $p$.
\end{lemma}

\begin{proof}
 Since $\Aut(C)$ acts transitively on $C$ on either side of the `if and only if' statement we may assume this holds. Since $\delta\geq 3$ it follows that every element of $C_1$ is distance $1$ from a unique codeword and hence $\{\varGamma_1(\alpha)\mid \alpha\in C\}$ forms a partition of $C_1$. In particular, any automorphism mapping an element $\beta_1\in C_1$ to $\beta_2\in C_1$ must also map the unique codeword $\alpha_1$ at distance $1$ from $\beta_1$ to the unique codeword $\alpha_2$ at distance $1$ from $\beta_2$. As $\Aut(C)$ acts transitively on $C$, we then have that $\Aut(C)_p$ acts transitively on the set of all lines incident with $p$ if and only if $\Aut(C)$ acts transitively on $C_1$, and the result holds.
\end{proof}

We now consider codes with minimum distance $\delta=4$.

\begin{lemma}
 Let $C$ be a neighbour-transitive code with minimum distance $\delta= 4$ in a generalised quadrangle $\Q$ of order $(s,t)$. Then each of the following hold:
 \begin{enumerate}
  \item $|C_1|=(t+1)|C|$.
  \item $|C|+|C_2|+|C_4|=(s+1)(st+1)$.
  \item $|C_1|+|C_3|=(t+1)(st+1)$.
 \end{enumerate}
\end{lemma}

\begin{proof}
 Since there are $t+1$ lines through each point, and $\delta\geq 4$ implies that each line is incident with at most one codeword, part 1 holds. The set of all points of $\Q$ has size $(s+1)(st+1)$ and is the disjoint union $C\cup C_2 \cup C_4$, giving part 2. The set of all  lines  of $\Q$ has size $(t+1)(st+1)$ and is the disjoint union $C_1\cup C_3$, which gives part 3. 
\end{proof}

The following lemma is the motivation for results of Section~\ref{ovoidssect}. 

\begin{lemma}\label{ovoidorspread}
 Let $C$ be a code in a generalised quadrangle $\Q$ with minimum distance $\delta= 4$ and covering radius $\rho$. Then the following hold:
 \begin{enumerate}
  \item $C$ is a partial ovoid or a partial spread of $\Q$.
  \item $C$ is a maximal partial ovoid or a maximal partial spread of $\Q$ if and only if $\rho\leq 3$.
  \item $C$ is an ovoid or spread of $\Q$ if and only if $\rho=2$.  
 \end{enumerate}
\end{lemma}

\begin{proof}
 By duality we may assume that there exists a point $p$ in $C$ and prove the result in terms of (maximal and/or partial) ovoids. Since $\delta$ is even and $\varGamma$ is bipartite it follows that $C$ is a subset of the set of points of $\Q$. Since $\delta= 4$, it follows that no line of $\Q$ is incident with two distinct points contained in $C$ and hence $C$ is a partial ovoid. Now, $\rho\leq 3$ if and only if $C_4$ is empty if and only if every point of $\Q$ is either in $C$ or lies on line contained in $C_1$ if and only if for every point $p'$ of $\Q$, with $p'\notin C$, the code $C\cup\{p'\}$ contains a pair of collinear points, that is, $C$ is a maximal partial ovoid. This proves part 2. Finally, $\rho=2$ if and only if $C_3$ is empty if and only if every line of $Q$ is incident with some element of $C$, proving part 3.
\end{proof}

The next result concerns neighbour-transitive codes in generalised quadrangles having an automorphism group that fixes a point. 

\begin{lemma}\label{pointstab}
 Suppose $C$ is a non-trivial neighbour-transitive code with minimum distance $\delta=3$ or $4$ in a generalised quadrangle $\Q$ and that $\Aut(C)$ fixes a point $p$ of $\Q$. Then $C\subseteq \varGamma_4(p)$.
\end{lemma}

\begin{proof}
 Since $p^G=\{p\}$ and we are assuming $|C|\geq 2$, we have that $p\notin C$. Also, since $\varGamma$ is bipartite and $\Aut(C)$ fixes $p$ but also must act transitively on $C$, we have that $C$ consists of either points or lines of $\Q$, but not both. Suppose that $C$ is a subset of the set of all points of $\Q$. Since $G$ fixes $p$ and preserves distances in $\varGamma$ it follows that either $C\subseteq\varGamma_2(p)$ or $C\subseteq\varGamma_4(p)$. Suppose there exists some $p_1\in C$ where $p_1\in\varGamma_2(p)$. By Lemma~\ref{lemma:Xptrans}, $G_{p_1}$ acts transitively on the set of lines through $p_1$. However, since $G$ fixes $p$, the stabiliser $G_{p_1}$ fixes the line through $p$ and $p_1$, giving a contradiction. Hence $C\subseteq \varGamma_4(p)$.
 
 Suppose $C$ is a subset of the set of all lines of $\Q$. Moreover, suppose $p\in C_1$. Then there exists some line $\ell\in C$ such that $p\in\varGamma_1(\ell)$. However, this implies that $G_\ell$ fixes $p$, which is in $\varGamma_1(\ell)$, contradicting Lemma~\ref{lemma:Xptrans}. Thus $p\in C_3$. By the defining property of a generalised quadrangle, there exists a unique point $p_1$ and a unique line $\ell_1$ such that $p$ lies on $\ell_1$ and $p_1$ lies on both $\ell$ and $\ell_1$. It follows that $G_\ell$ fixes the point $p_1$, again contradicting Lemma~\ref{lemma:Xptrans}. Thus the result holds.
\end{proof}

%
%

\section{Ovoids and spreads}\label{ovoidssect}

In this section we classify all the neighbour-transitive ovoids and spreads in thick classical generalised quadrangles.

\begin{lemma}\label{lemma:ovoiddivis}
 Let $C$ be a $G$-neighbour-transitive ovoid in a generalised quadrangle $\Q$ of order $(s,t)$, and $p\in C$. Then $G_p$ has order divisible by $t+1$ and $G$ has order divisible by $(t+1)(st+1)$.
\end{lemma}

\begin{proof}
 By \cite[1.8.1]{paynefinite}, $C$ has size $(st+1)$. Since $G$ acts transitively on $C$, it follows that $(st+1)$ divides $|G|$. If $p\in C$ then, by Lemma~\ref{lemma:Xptrans}, $G_p$ acts transitively on the set of lines incident with $p$. There are $t+1$ lines through $p$, so that $t+1$ divides $|G_p|$. Hence the result holds.
\end{proof}

The following three lemmas provide examples of neighbour-transitive codes in generalised quadrangles.

\begin{lemma}\label{regspreadisNT}
 A regular spread of $\W_3(q)$ is neighbour-transitive.
\end{lemma}

\begin{proof}
 Let $C$ be a regular spread of $\W_3(q)$. Then $\Aut(C)\cong \PGaL_2(q^2)$ acts $2$-transitively on the codewords of $C$ (see \cite[Sections~3.2 and 3.4]{bamberg2009classification}). Moreover,  $\Aut(C)_\ell=q^2:\GaL_1(q^2)$ acts transitively on the set of points incident with $\ell$. Thus $C$ is neighbour-transitive.
\end{proof}

\begin{lemma}\label{regspreadminusline}
 A regular spread of $\W_3(q)$ with a line removed is a neighbour-transitive code of size $q^2$.
\end{lemma}

\begin{proof}
 Let $C$ be a regular spread of $\W_3(q)$ with one line removed. Then $C$ has size $q^2$ and, since $\PGaL_2(q^2)$ acts $2$-transitively on the lines of a regular spread (see \cite[Sections~3.2 and 3.4]{bamberg2009classification}), we have that $\Aut(C)\cong q^2:\GaL_1(q^2)$ acts transitively on $C$. It remains to prove that for a line $\ell\in C$, the stabiliser $\Aut(C)_\ell$ acts transitively on the set of points on $\ell$. Indeed $\Aut(C)_\ell\cong \GaL_1(q^2)$ acts transitively on the set of points incident with $\ell$. Thus $C$ is neighbour-transitive.
\end{proof}

\begin{lemma}\label{classicalunital}
 A classical ovoid of $\H_3(q^2)$ is neighbour-transitive. 
\end{lemma}

\begin{proof}
 Let the points and lines of $\H_3(q^2)$ be the totally isotropic rank $1$ and $2$, respectively, subspaces of the vector space $V=\langle e_1,e_1,e_3,e_4\rangle$ over $\F_{q^2}$ under the quadratic form $f(x)=x_1^{q+1}+x_2^{q+1}+x_3^{q+1}+x_4^{q+1}$. The subgroup of $\PSU_4(q)$ stabilising a non-singular point $P$ is $\PSU_3(q)$ and the set of singular points contained in $P^\perp$ make up a classical ovoid $C$ of $\H_3(q^2)$, on which $\PSU_3(q)$ acts $2$-transitively \cite[Section~3.1]{bamberg2009classification}. Thus, in order to prove that $C$ is neighbour-transitive all that is left to prove is that for some $p\in C$ we have that $\Aut(C)_p$ acts transitively on the lines through $p$.
 
 Let $P=\langle e_4\rangle$, so that $P^\perp=\langle e_1,e_2,e_3\rangle$, and let $p=\langle e_1+be_2\rangle$, where $b\in\F_q^2$ is a fixed element chosen so that $b^{q+1}=-1$ (note that such an element $b$ always exists since we are working in the quadratic extension $\F_{q^2}$ of $\F_q$). Each line through $p$ then has the form $\langle p, ae_3+abe_4\rangle$ where $a\in\F_{q^2}$ such that $a^{q+1}=1$; note that there are precisely $q+1$ choices for $a$ here. Now, the map $g_\lambda:(e_1,e_2,e_3,e_4)\mapsto (e_1,e_2,\lambda e_3,\lambda e_4)$, where $\lambda\in\F_{q^2}$, is an element of $\Aut(C)_p$ precisely when $\lambda^{q+1}=1$. Moreover, $g_{a^{-1}}$ maps $\langle p, ae_3+abe_4\rangle$ to $\langle p, e_3+be_4\rangle$, and hence $\Aut(C)_p$ acts transitively on the set of lines through $p$. Thus $C$ is neighbour-transitive.
\end{proof}

We can now prove the main result of this section. 

\begin{theorem}\label{NTovoidsspreads}
 Let $C$ be a neighbour-transitive code with minimum distance $4$ and covering radius $\rho=2$ in a thick classical generalised quadrangle $\Q$ of order $(s,t)$ and assume that $\Aut(C)$ is insoluble. Then $C$ is equivalent to one of the following:
 \begin{enumerate}
  \item The regular spread of $\W_3(q)$.
  \item A classical ovoid of $\H_3(q^2)$.
 \end{enumerate}
\end{theorem}

\begin{proof}
 Recall that a generalised quadrangle is polar space of rank $2$. Since $G$  acts transitively on $C$, we can apply \cite[Theorem~1.1]{bamberg2009classification}. As we are only interested in those polar spaces appearing in Table~\ref{table:classGQs}, we then have that \cite[Table~1]{bamberg2009classification} implies that $C$ is equivalent to either (i) a regular spread of $\W_3(q)$, (ii) a classical ovoid of $\H_3(q^2)$, (iii) the ``exceptional ovoid'' of $\H_3(5^2)$ or (iv) the Suzuki-Tits ovoid of $\Qu_4(2^{2h+1})$. Note that many of the examples listed in \cite[Table~1]{bamberg2009classification} are the duals of another entry and thus are equivalent as codes in the point-line incidence graphs. 
 
 The codes (i) and (ii) are neighbour-transitive by Lemmas~\ref{regspreadisNT} and~\ref{classicalunital}. The full automorphism group of the exceptional ovoid of $\H_3(5^2)$ is given in \cite[Section~3.1]{bamberg2009classification} and has order $2^5\cdot 3^2\cdot 7$; this is not divisible by $(t+1)(st+1)=2^2\cdot 3^3\cdot 7$, and hence, by Lemma~\ref{lemma:ovoiddivis}, this ovoid is not neighbour-transitive. In case (iv), $\Aut(C)$ is contained in the automorphism group of the Suzuki group $\Sz(q)$. Hence $|G|$ divides $d(p^{2d}+1)p^{2d}(p^d-1)$, where $q=p^d$, and is thus not divisible by $(p^d+1)(p^{2d}+1)$, contradicting Lemma~\ref{lemma:ovoiddivis}. This completes the proof.
\end{proof}

\section{\texorpdfstring{Codes in ${\mathsf W}_3(q)$}{W(3,q)}}\label{w3qsect}

In this section we consider codes in the classical generalised quadrangle $\W_3(q)$ with points an lines given, respectively, by isotropic $1$-dimensional and $2$-dimensional subspaces under a symplectic bilinear form. First we consider neighbour-transitive codes of size $2$. Note that, since $\W_3(q)$ is self-dual if and only if $q$ is even, when $q$ is even a pair of lines may be mapped to a pair of points under the duality automorphism, but not when $q$ is odd.
\begin{proposition}
 Let $C$ be a code of size $2$ with minimum distance $\delta=4$ in ${\mathsf W}_3(q)$. Then $C$ is equivalent to either a pair of non-collinear points or a pair of non-intersecting lines and $C$ is neighbour-transitive.
\end{proposition}

\begin{proof}
 Since $\delta=4$, we have that $C$ is either a pair of non-collinear points or a pair of non-intersecting lines. The automorphism group $G$ of the point-line incidence graph $\varGamma$ of $\W_3(q)$ contains $\PSp_4(q)$ which acts transitively on the set of all pairs of non-collinear points as well as transitively on the set of all pairs of non-intersecting lines. Let $C$ be a pair of non-collinear points (respectively, non-intersecting lines). The stabiliser in $\Aut(C)$ of a point $p$ (respectively, line $\ell$) in $C$ contains a subgroup isomorphic to $\SL_2(q)$ that acts transitively on the set of all lines incident with $p$ (points incident with $\ell$). Thus $C$ is neighbour-transitive.
\end{proof}

In the following few results we construct the codes arising in Theorem~\ref{mainresult}.

\begin{proposition}\label{hyplineNT}
 Let $C$ be the set of points contained in a hyperbolic line of $\W_3(q)$. Then $C$ is a neighbour-transitive maximal partial ovoid.
\end{proposition}

\begin{proof}
 First, by \cite[Theorem~2.1]{cimrakova2007smallest}, $C$ is a maximal partial ovoid. Following \cite[Section~3.5.4]{wilson2009finite}, let $V\cong\F_q^4$ be the underlying vector-space of $\W_3(q)$ equipped with a symplectic bilinear form $f$. If $W$ is a non-singular rank $2$ subspace of $V$ then it follows that $V=W\oplus W^\perp$ and that the stabiliser inside $\Sp_4(q)$ of $W$ is $\SL_2(q)\times\SL_2(q)$. This induces a group that acts transitively on the set of all pairs $(p_1,p_2)$ where $p_1$ is a point contained in $W$ and $p_2$ is a point contained in $W^\perp$. Hence $\Aut(C)$ acts transitively on $C$. Moreover, each point $p_1$ in $W$ is collinear with each point $p_2$ in $W^\perp$, which implies that the set of $q+1$ lines through $p_1$ is in bijection with the set of $q+1$ points in $W^\perp$. Thus $\Aut(C)_{p_1}$ acts transitively on the set of lines through $p_1$ and the result holds.
\end{proof}


\begin{example}\label{smallqequals3}
 Let $V\cong\F_3^4$ be equipped with the symplectic form $f(x,y)=x_1y_2-x_2y_1+x_3y_4-x_4y_3$. Representing each line of $\W_3(3)$ as a $2\times 4$ matrix via its row-space, let
 \[
  C=\Big\{ 
  \begin{bmatrix}
   1 & 0 & 1 & 0\\
   1 & 1 & 2 & 2
  \end{bmatrix},
  \begin{bmatrix}
   1 & 0 & 0 & 1\\
   1 & 1 & 1 & 2
  \end{bmatrix},
  \begin{bmatrix}
   1 & 0 & 2 & 1\\
   1 & 1 & 0 & 1
  \end{bmatrix},
 \]
 \vspace{-0.25cm}
 \[
  \begin{bmatrix}
   1 & 0 & 1 & 2\\
   1 & 1 & 1 & 1
  \end{bmatrix},
  \begin{bmatrix}
   1 & 0 & 2 & 2\\
   1 & 1 & 2 & 0
  \end{bmatrix}
  \Big\}.
 \]
 Then $C$ has automorphism group $\s_5$ and, up to equivalence, $C$ is the unique neighbour-transitive code of size $5$ with minimum distance $4$ in $\W_3(3)$, as confirmed in MAGMA. Moreover $C$ has covering radius $3$, and is thus a maximal partial spread of $\W_3(3)$.
\end{example}

Next we give a construction for maximal partial spreads of $\W_3(q)$ of size $q^2-1$. Note that, by duality, such a partial spread is equivalent to a maximal partial ovoid of $\Qu_4(q)$ of size $q^2-1$ and the construction and examples given in Example~\ref{maxspreadconstr} are hence equivalent to those given in \cite{coolsaet2013known} (see, in particular, Theorem~4 and the Table on page 5 of \cite{coolsaet2013known}). Note also that for $q=5,7,11$ these examples were first found by Penttila with the aid of a computer. It is a conjecture of Thas \cite[Section~1.9.2]{thas2004symmetry} that these examples comprise all of the maximal partial ovoids of $\Qu_4(q)$ having size $q^2-1$.

\begin{example}\label{maxspreadconstr}
 The following is the dual of the construction considered in \cite{coolsaet2013known}. Let $V\cong\F_q^4$ be equipped with the symplectic form $f(x,y)=x_1y_2-x_2y_1-x_3y_4+x_4y_3$. Let $G$ be, as in Table~\ref{table:sharplytransgps}, a sharply transitive subgroup of $\SL_2(q)$ (represented as $2\times 2$ matrices) and let $C$ be the set of rank $2$ subspaces of $V$ (represented via the row-spaces of $2\times 4$ matrices)
 \[
  C=\{\begin{bmatrix} I & A \end{bmatrix}\mid A\in G\},
 \]
 where $I$ is the $2\times 2$ identity matrix. The row-space of $\begin{bmatrix} I & A \end{bmatrix}$ will be an isotropic $2$-space if and only if ${\rm det} I - {\rm det} A=0$, that is, ${\rm det} A=1$. Since $A\in \SL_2(q)$, we have that $C$ is indeed a set of lines of $\W_3(q)$.
\end{example}

\begin{table}
 \begin{center}
 \begin{tabular}{cc}
  $q$ & $G$ \\
  \hline
   $2$ & $\GL_1(4)$\\
   $3$ & $Q_8$\\
   $5$ & $2.\alt_4$\\
   $7$ & $2.\s_4$\\
   $11$ & $\SL_2(5)$\\
  \hline
 \end{tabular}
 \caption{Subgroups of $\SL_2(q)$ of order $q^2-1$ (see \cite[Chapter 3, Section 6]{suzuki1982group}). }
 \label{table:sharplytransgps}
 \end{center}
\end{table}

\begin{lemma}
 Let $C$ be one of the examples in Example~\ref{maxspreadconstr}, where $G$ is one of the groups listed in Table~\ref{table:sharplytransgps}. Then $C$ is a neighbour-transitive maximal partial spread of $\W_3(q)$.
\end{lemma}

\begin{proof}
 To see that $C$ is a partial spread, let $A$ and $B$ be distinct elements of $G$ and consider the matrix
 \[
  M=
  \begin{bmatrix}
   I & A\\
   I & B
  \end{bmatrix}.
 \]
 Now, $\begin{bmatrix} I & A \end{bmatrix}$ and $\begin{bmatrix} I & B \end{bmatrix}$ represent non-intersecting lines if and only if $M$ has full rank, which occurs if and only if $A-B$ has rank $2$. It thus follows from \cite[Lemma~1~(iv)]{coolsaet2013known} that the construction given in that paper is equivalent to the one given here. Hence, by \cite[Theorem~4]{coolsaet2013known}, $C$ is a maximal partial spread. Let
 \[
  g_{X,Y}=
  \begin{bmatrix}
   X & O\\
   O & Y
  \end{bmatrix},
 \]
 where $O$ is the $2\times 2$ matrix having all entries equal to $0$, and $X,Y\in G$. Then $g_{X,Y}$ is an automorphism of $\W_3(q)$ acting via right multiplication and maps $\begin{bmatrix} I & A \end{bmatrix}$ to $\begin{bmatrix} X & AY \end{bmatrix}$, which represents the same line as $\begin{bmatrix} I & X^{-1}AY \end{bmatrix}$. Hence, $g_{I,A^{-1}}$ maps the line $\begin{bmatrix} I & A \end{bmatrix}$ to $\begin{bmatrix} I & I \end{bmatrix}$, which implies that $\Aut(C)$ acts transitively on $C$. Also $g_{X,X}$ stabilises $\begin{bmatrix} I & I \end{bmatrix}$ and, since $G$ is sharply transitive, $X$ may be chosen so that $g_{X,X}$ maps any given vector $(a,b,a,b)$ in the row-space of $\begin{bmatrix} I & I \end{bmatrix}$ to $(1,0,1,0)$. Thus $\Aut(C)$ acts transitively on $C_1$ and hence $C$ is neighbour-transitive.
\end{proof}

\begin{lemma}\label{sharplytransNT}
 Suppose that $C$ is a neighbour-transitive code with minimum distance $4$ having size $q^2-1$ in $\W_3(q)$. Then $q=2,3,5,7$ or $11$ and $C$ is equivalent to a maximal partial spread as in Example~\ref{maxspreadconstr}.
\end{lemma}

\begin{proof}
 First, $C$ is either a partial spread or a partial ovoid of $\W_3(q)$. If $q$ is even then, by duality, we may assume that $C$ is a partial spread. If $q$ is odd then, by \cite{tallini1991blocking}, a partial ovoid has size at most $q^2-q+1$, which is less than $q^2-1$ for all $q\geq 3$, and hence we may again assume that $C$ is a partial spread. The dual of the main result of \cite{de2013large} then implies that $q$ is prime. By \cite[Theorem~5]{coolsaet2013known} we have that $C$ as in Example~\ref{maxspreadconstr} except that $G$ may be a subset, rather than a subgroup, of $\SL_2(q)$. By the dual of \cite[Theorem~2]{coolsaet2013known}, $\Aut(C)$ fixes the sub-quadrangle of $\W_3(q)$ consisting of the dual grid induced by the pair of lines $\begin{bmatrix} I & O \end{bmatrix}$ and $\begin{bmatrix} O & I \end{bmatrix}$. Hence, $q$ prime implies that $\Aut(C)$ is contained in the group $\SL_2(q)\wr \s_2$ generated by the elements 
 \[
  g_{X,Y}=
  \begin{bmatrix}
   X & O\\
   O & Y
  \end{bmatrix},
 \]
 where $X,Y\in \SL_2(q)$, and the involution 
 \[
  \begin{bmatrix}
   O & I\\
   I & O
  \end{bmatrix}.
 \]
 Since we may assume that $\begin{bmatrix} I & I \end{bmatrix}$ is in $C$, and $C$ is then the orbit of $\begin{bmatrix} I & I \end{bmatrix}$ under $\Aut(C)$, it follows that $G$ must be a group. Hence the result holds.
\end{proof}

We may now prove Theorem~\ref{mainresult}.

\begin{proof}[Proof of Theorem~\ref{mainresult}]
 First, by \cite[Theorem~4.3 and Remark~4.4]{de2005partial}, if $C$ has size $q+1$ then $C$ consists of the points of a hyperbolic line, and, by Proposition~\ref{hyplineNT}, such a code is neighbour-transitive. Hence part 4 holds. Also, part 5 is exhibited by Example~\ref{smallqequals3}. If $|C|=q^2-1$ then, by Lemma~\ref{sharplytransNT}, part 3 holds. Thus we may assume that $C$ has size at least $q^2$. It then follows from \cite[2.7.1]{paynefinite} that $C$ is contained in a uniquely defined spread. If $C$ has size $q^2$ then part $2$ holds, and if $C$ is a spread then, by Theorem~\ref{NTovoidsspreads}, $C$ is a regular spread, and part 1 holds.
\end{proof}

\end{document}